\documentclass[12pt]{amsart}
\usepackage{latexsym}

\usepackage{amsmath}

\newcommand{\ignore}[1]{}

\newcommand{\hide}[1]{}

\DeclareMathOperator{\Li}{Li}

\newcommand{\F}{\mathbb F}

\newcommand{\Q}{\mathbb Q}

\newtheorem{dummy}{Dummy}

\newtheorem{lemma}[dummy]{Lemma}

\newtheorem{theorem}[dummy]{Theorem}

\newtheorem{cor}[dummy]{Corollary}

\theoremstyle{definition}

\newtheorem{question}[dummy]{Question}

\theoremstyle{remark}
\newtheorem{rem}[dummy]{Remark}
\newtheorem*{rem*}{Remark to ourselves}

\begin{document}

\bibliographystyle{amsalpha}

\author{Marina Avitabile}
\email{marina.avitabile@unimib.it}
\address{Dipartimento di Matematica e Applicazioni\\
  Universit\`a degli Studi di Milano - Bicocca\\
 via Cozzi 55\\
  I-20125 Milano\\
  Italy}
\author{Sandro Mattarei}
\email{smattarei@lincoln.ac.uk}
\address{School of Mathematics and Physics \\
University of Lincoln \\
Brayford Pool
Lincoln, LN6 7TS\\
United Kingdom}
\title[A generalized truncated logarithm]{A generalized truncated logarithm}


\begin{abstract}
We introduce a generalization $G^{(\alpha)}(X)$ of the truncated logarithm $\pounds_1(X)=\sum_{k=1}^{p-1}X^k/k$ in characteristic $p$,
which depends on a parameter $\alpha$.
The main motivation of this study is $G^{(\alpha)}(X)$ being an inverse, in an appropriate sense, of a parametrized generalization of the truncated exponential
given by certain Laguerre polynomials.
Such Laguerre polynomials play a role in a {\em grading switching} technique for non-associative algebras,
previously developed by the authors, because they satisfy a weak analogue of the functional equation
$\exp(X)\exp(Y)=\exp(X+Y)$ of the exponential series.
We also investigate functional equations satisfied by $G^{(\alpha)}(X)$ motivated by known functional equations for $\pounds_1(X)=-G^{(0)}(X)$.
\end{abstract}

\maketitle

\section{Introduction}\label{sec:intro}

In this paper we investigate a parametrized analogue of the logarithmic function in prime characteristic.
This modified version of the logarithm arises as the inverse (in an appropriate sense) of a parametrized modular analogue
of the exponential which was introduced in~\cite{AviMat:Laguerre}.
To provide motivation to this work we give here only a summary introduction to the latter, referring the interested reader
to~\cite{AviMat:Laguerre,AviMat:gradings} for a detailed description.

The usefulness of the exponential function in various part of mathematics ultimately stems from its property
$e^x\cdot e^y=e^{x+y}$, sometimes disguised as its essentially equivalent differential formulation $(d/dx)e^x=e^x$.
In particular, as a means of connecting additive and multiplicative structures the exponential has always played a role in Lie theory,
for local reconstruction of a Lie group from its Lie algebra.
The essence of this application is that, under appropriate conditions, {\em the exponential of a derivation (of a non-associative algebra) is an automorphism}.
The simplest setting to see this mechanism in action is the following:
if $D$ is a nilpotent derivation of a
non-associative algebra $A$ over a field of characteristic zero, then the finite sum
$\exp(D)=\sum_{k=0}^{\infty}D^k/k!$ defines an automorphism of $A$.
This can be viewed as a formal consequence of the functional equation $\exp(X)\cdot\exp(Y)=\exp(X+Y)$ satisfied by the exponential series
$\exp(X)=\sum_{k=0}^{\infty}X^k/k!$, in the ring of formal power series $\Q[[X,Y]]$.

Over fields prime characteristic $p$, the {\em truncated exponential} $E(X)=\sum_{k=0}^{p-1}X^k/k!$
can be used as a substitute for the exponential series, to some extent.
In particular, if $p$ is odd and the derivation $D$ satisfies $D^{(p+1)/2}=0$, then $E(D)$ is an automorphism,
but the weaker condition $D^p=0$ is insufficient to this goal.
This apparent shortcoming of $\exp(D)$ was turned into an advantage in the technique of
{\em toral switching} for modular Lie algebras, originally developed by Winter~\cite{Win:toral} in its most basic form.
This fundamental tool in the classification theory of simple modular Lie algebras
has later undergone substantial generalizations by Block and Wilson~\cite{BlWil:rank-two},
until reaching its most complete expression in work of Premet~\cite{Premet:Cartan}.
The gist of the technique is that maps similar to exponentials of derivations
(of the form $E(D)$ with $D$ a certain inner derivation, or more general constructions
in~\cite{Premet:Cartan}) are used to produce a new torus from a given torus
(with certain properties including having maximal dimension).
Because those maps are not automorphims, the new torus may
have very different properties from the original torus, and better suited to certain purposes.

Now, any torus of a Lie algebra induces a grading given by the corresponding eigenspace
decompositions with respect to the adjoint action (called a {\em generalized root space decomposition}).
Therefore, in very broad terms toral switching may be viewed as a technique to pass from a given grading of a Lie algebra to another one.
In such generality one may wonder whether some kind of exponential could be used
to pass from a given grading to another one without assuming that either of them is associated to some torus.
In this spirit a {\em grading switching} technique
for non-associative algebras was developed in~\cite{Mat:Artin-Hasse}
for nilpotent derivations, and then extended in~\cite{AviMat:Laguerre} for arbitrary derivations.
Despite being motivated and strongly inspired by the toral switching technique, to attain greater generality
and allow applications such as those in~\cite{Mat:Artin-Hasse,AviMat:-1,AviMat:A-Z,AviMat:gradings}, those papers focus
on the algebraic property which makes some generalized exponential of a derivation map a grading of the algebra to another grading.
That crucial property boils down to the generalized exponential employed satisfying an appropriately weakened version
of the functional equation $\exp(X)\cdot\exp(Y)=\exp(X+Y)$.

In fact, the generalized exponential employed in~\cite{Mat:Artin-Hasse} was the Artin-Hasse exponential series
\[
E_p(X):=\Bigl(\sum_{i=0}^{\infty}X^{p^i}/p^i\Bigr)=\prod_{i=0}^{\infty}\exp(X^{p^i}/p^i),
\]
whose importance in $p$-adic analysis and number theory stems from all its coefficients being $p$-integral.
In particular, by viewing its coefficients modulo $p$ one can regard $E_p(X)\in\F_p[[X]]$, where $\F_p$ is the field of $p$ elements.
As such, $E_p(X)$ was shown in~\cite{Mat:Artin-Hasse} to have the property that all terms of
$
E_p(X)\cdot E_p(Y)/E_p(X+Y)\in\F_p[[X,Y]]
$
have degree multiples of $p$.
This weak version of the fundamental functional equation for $\exp(X)$ is precisely what make grading switching based on $E_p(D)$ work,
for appropriate gradings and nilpotent derivations $D$.
It was shown in~\cite{Mat:exponential} that this weak functional equation actually characterizes $E_p(X)$ in $\F_p[[X]]$
up to certain natural variations.

Nilpotence of the derivation $D$ was required in~\cite{Mat:Artin-Hasse} for $E_p(D)$ to make sense algebraically.
However, the resulting grading switching matched only a rather special case of the classical toral switching.
This limitation was removed in~\cite{AviMat:Laguerre}
by introducing certain (generalized) Laguerre polynomials as substitutes for exponentials.
We defer discussing the connection with the classical Laguerre polynomials $L_n^{(\alpha)}(X)$ to Subsection~\ref{subsec:Laguerre}, but those of interest to us,
once regarded as polynomials in characteristic $p$, take the form
\[
L_{p-1}^{(\alpha)}(X)
=(1-\alpha^{p-1})
\sum_{k=0}^{p-1}\frac{X^k}{(1+\alpha)(2+\alpha)\cdots(\alpha+k)}.
\]
Here $\alpha$ is a parameter, which in the application to grading switching relates to an eigenvalue of the derivation $D$ in a certain way.
In the special case where the derivation $D$ is nilpotent its only eigenvalue is zero, and
$L_{p-1}^{(0)}(X)$ equals the truncated exponential $E(X)$.
The reader who might be puzzled by how a construction using certain Laguerre {\em polynomials} can extend
an earlier one based on Artin-Hasse {\em power series} will find an explanation of this connection in~\cite[Proposition~6]{AviMat:gradings}.

Again, the crucial property which makes $L_{p-1}^{(\alpha)}(X)$ work as a replacement for an exponential
in the grading switching application in~\cite{AviMat:Laguerre} is that it satisfies a weak version
of the fundamental functional equation for $\exp(X)$, its most important feature being that all terms
$
L_{p-1}^{(\alpha)}(X)\cdot L_{p-1}^{(\beta)}(Y)/L_{p-1}^{(\alpha+\beta)}(X)\in\F_p(\alpha,\beta)[[X,Y]]
$
have degree multiples of $p$.
The more precise version of this property given in Theorem~\ref{thm:functional-Laguerre-simplified} below is actually required,
and an even more precise one as in~\cite[Proposition~2]{AviMat:Laguerre} to attain full generality.
We skip those details in this introduction, but only mention that the functional equation in more precise form
is actually a congruence involving moduli such as $X^p-(\alpha^p-\alpha)$ and $Y^p-(\beta^p-\beta)$.
In fact, in this context $L_{p-1}^{(\alpha)}(X)$ is only of interest when regarded modulo $X^p-(\alpha^p-\alpha)$.

We come now to the goal of this paper, which is investigating an appropriate compositional inverse
$G^{(\alpha)}(X)$ for the polynomial $L_{p-1}^{(\alpha)}(X)$.
We devote Section~\ref{sec:prel} to a detailed description of our results, along with the necessary technical preliminaries,
hence we limit ourselves to a succint outline in this Introduction.
Guidance from the functional properties of $L_{p-1}^{(\alpha)}(X)$ suggests that the correct interpretation
of being a left compositional inverse is being a polynomial of degree less than $p$ in $X$, and with coefficients depending on the parameter $\alpha$,
satisfying
\[
G^{(\alpha)}\bigl(L_{p-1}^{(\alpha)}(X)\bigr)\equiv X \pmod{X^p-(\alpha^p-\alpha)}.
\]
General reasons imply that there is precisely one polynomial $G^{(\alpha)}(X)$ satisfying these requirements.
In Theorem~\ref{thm:inverse} we find an explicit expression for it, but here we limit ourselves to pointing out that
specialized to $\alpha=0$ it becomes $G^{(0)}(X)=-\sum_{k=1}^{p-1}X^k/k$, which is a truncated version of the series for $\log(1-X)$.
In terms of the standard notation $\pounds_d(X)=\sum_{k=1}^{p-1}X^k/k^d$ for the {\em finite polylogarithms} we then have
$G^{(0)}(X)=-\pounds_1(X)$.
The above equation relating $L_{p-1}^{(\alpha)}(X)$ and $G^{(\alpha)}(X)$, which reads
$-\pounds_1\bigl(E(X)\bigr)\equiv X\pmod{X^p}$
when $\alpha=0$, where $E(X)$ is the truncated exponential, may be viewed as an analogue of
$\log\bigl(\exp(X)\bigr)=X$.
This depends on the fact that the {\em truncated logarithm} $\pounds_1(X)$ (always viewed in characteristic $p$ here) satisfies $\pounds_1(1-X)=\pounds_1(X)$,
a functional equation without analogue in characteristic zero, as we explain in the discussion which follows Theorem~\ref{thm:inverse}.
The polynomial $G^{(\alpha)}(X)$ is also a right compositional inverse of $L_{p-1}^{(\alpha)}(X)$ in an appropriate sense,
made explicit in Corollary~\ref{cor:inverse}.
The proofs of Theorem~\ref{thm:inverse} and Corollary~\ref{cor:inverse} occupy Section~\ref{sec:proofs}.

The coefficients of $G^{(\alpha)}(X)$ as given in Theorem~\ref{thm:inverse} are rational expressions in $\alpha$ with numerators $1$.
We devote Section~\ref{sec:roots} to an investigation of their denominators, leading to an explicit description of their factorizations, as polynomials in $\alpha$.
In particular, it will turn out that their roots are nonzero elements of the prime field $\F_p$, as stated in Theorem~\ref{thm:inverse-more},
but Theorem~\ref{thm:roots} allows explicit calculation of which roots occur and with which multiplicities.

This information will be crucial in the next part of the paper, which studies functional equations for $G^{(\alpha)}(X)$.
This investigation is motivated by functional equations for the truncated logarithm $\pounds_1(X)=-G^{(0)}(X)$, which besides the already mentioned
$\pounds_1(1-X)=\pounds_1(X)$ satisfies $\pounds_1(X)=-X^p\pounds_1(1/X)$.
In Theorem~\ref{thm:reciprocal} we generalize the latter to an equation for $G^{(\alpha)}(X)$,
which involves an evaluation of $G^{(\alpha)}$ on one side and one of $G^{(-\alpha)}$ on the other.
It remains unclear whether a similar generalization exists for $\pounds_1(1-X)=\pounds_1(X)$, and even which general form it may take.

More interestingly, the functional equation $\log(xy)=\log(x)+\log(y)$ for the ordinary logarithmic function,
which matches the fundamental functional equation for the exponential function,
has a well-known analogue for the truncated logarithm $\pounds_1(X)$ in the polynomial ring $\F_p[X,Y]$,
except that the latter involves four terms, see Equation~\eqref{eq:4-term}.
We are unfortunately unable to generalize this to a corresponding equation for $G^{(\alpha)}(X)$, but in Theorem~\ref{thm:powers}
we produce a version for $G^{(\alpha)}(X)$ of a consequence of $\log(xy)=\log(x)+\log(y)$, namely $\log(x^h)=h\log(x)$ for integer $h$.
Our equation, which is actually a congruence with respect to an appropriate modulus,
involves an evaluation of $G^{(\alpha)}$ on one side and one of $G^{(h\alpha)}$ on the other, for $0<h<p$.
We discuss these functional equations extensively in Subsection~\ref{subsec:functional},
and prove our results in Section~\ref{sec:proofs_B}.

In the final Section~\ref{sec:Jacobi} we relate polynomials used to describe the coefficients of $G^{(\alpha)}(X)$ in Theorem~\ref{thm:inverse}
to certain Jacobi polynomials, showing how some of the properties of the former could alternately be deduced
by appropriate manipulation of known equations for the latter.

\section{Motivations and statements of the results}\label{sec:prel}

In this section we state our results in more precise form than the
rough description given in Section~\ref{sec:intro}.
In Subsection~\ref{subsec:Laguerre} we recall certain generalized Laguerre polynomials $L_{p-1}^{(\alpha)}(X)$ which play a role in the {\em grading switching}
described in~\cite{AviMat:Laguerre} because of their exponential-like property.
That crucial property is a functional equation, in the form of a congruence,
which we quote from~\cite{AviMat:Laguerre} in Theorem~\ref{thm:functional-Laguerre-simplified}.
In Subsection~\ref{subsec:inverting} we take this as a motivation to produce a certain compositional inverse $G^{(\alpha)}(X)$ of $L_{p-1}^{(\alpha)}(X)$,
in a suitable sense, which we describe in Theorem~\ref{thm:inverse}, and should have logarithm-like properties.
The short Subsection~\ref{subsec:more} aims at a better understanding of the coefficients of $G^{(\alpha)}(X)$,
namely their factorizations as rational functions of $\alpha$.
That detailed information is essential for the results of Subsection~\ref{subsec:functional}, which discusses functional equations for the truncated logarithm,
and extends a couple of them to equations for $G^{(\alpha)}(X)$.

\subsection{Some generalized Laguerre polynomials}\label{subsec:Laguerre}
The classical (generalized) Laguerre polynomial of degree $n \geq 0$ is defined as
\[
L_n^{(\alpha)}(X)=\sum_{k=0}^n\binom{\alpha+n}{n-k}
\frac{(-X)^k}{k!},
\]
where $\alpha$ is a parameter, usually taken in the complex numbers.
However, we may also view $L_n^{(\alpha)}(X)$ as a polynomial with rational coefficients in the two
indeterminates $\alpha$ and $X$, hence in the polynomial ring $\Q[\alpha,X]$.

Now fix a prime $p$.
If $0\le n<p$ then all coefficients of $L_n^{(\alpha)}(X)$ are $p$-integral, and hence can be viewed modulo $p$.
We are essentially only interested in the case $n=p-1$, where $L_{p-1}^{(\alpha)}(X)$ modulo $p$
can be considered as a generalization of the {\em truncated exponential}
$E(X)=\sum_{k=0}^{p-1}X^k/k!$
which we mentioned in the Introduction.
In fact, we have
$L_{p-1}^{(0)}(X)\equiv E(X)\pmod{p}$
because
$\binom{p-1}{k}\equiv\binom{-1}{k}=(-1)^k\pmod{p}$
for $0\le k<p$.

It is notationally convenient to work directly in characteristic $p$,
and so we will regard $L_{p-1}^{(\alpha)}(X)$ as having coefficients in the field $\F_p$ of $p$ elements,
thus viewing
\begin{equation}\label{eq:lag_p-1}
L_{p-1}^{(\alpha)}(X)
=\sum_{k=0}^{p-1}\binom{\alpha-1}{p-1-k}\frac{(-X)^k}{k!}
\in\F_p[\alpha,X].
\end{equation}
Hence $L_{p-1}^{(0)}(X)=E(X)$, if we regard the truncated exponential $E(X)$ as a polynomial in $\F_p[X]$ as well.
Taking advantage of working in characteristic $p$, and in particular of the congruence $k!(p-1-k)!\equiv(-1)^{k-1}\pmod{p}$,
we find simpler expressions for $L_{p-1}^{(\alpha)}(X)$
by expressing the binomial coefficients involved in terms of {\em Pochhammer symbols}
$(x)_m:=x(x-1)\cdots(x-m+1)$,
which are defined for $x$ in any commutative ring and $m$ a nonnegative integer, with the natural convention $(x)_0:=1$.
Thus, we have
\[
L_{p-1}^{(\alpha)}(X)
=-\sum_{k=0}^{p-1}(\alpha-1)_{p-1-k}\cdot X^k
=(1-\alpha^{p-1})
\sum_{k=0}^{p-1}\frac{X^k}{(\alpha+k)_{k}}.
\]
The latter expression, which is obtained using $(\alpha-1)_{p-1}=\alpha^{p-1}-1$ in $\F_p[\alpha]$,
emphasizes the role of $L_{p-1}^{(\alpha)}(X)$ as a generalization of the more familiar $E(X)$.
In particular, note that as a polynomial in $X$ only, its constant term is
$L_{p-1}^{(\alpha)}(0)=1-\alpha^{p-1}$.

The Laguerre polynomials $L_{p-1}^{(\alpha)}(X)$ satisfy a congruence which bears a strong analogy with the
functional equation $\exp(X)\exp(Y)=\exp(X+Y)$ of the classical exponential.
More precisely, viewing the functional equation for the exponential as an identity in the ring of power series
$\Q[[X,Y]]$, its direct consequence $\exp(X)\exp(Y)\equiv\exp(X+Y)\pmod{(X^p,Y^p)}$ in $\Q[[X,Y]]$
is equivalent to $E(X)E(Y)=E(X+Y)\pmod{(X^p,Y^p)}$ in the polynomial ring $\Q[X,Y]$.
Because no denominator in this equation is a multiple of $p$, the equation can be viewed modulo $p$, that is,
in the polynomial ring $\F_p[X,Y]$.
As such, it generalizes as follows for $L_{p-1}^{(\alpha)}(X)$.

\begin{theorem}[{\cite[Proposition~2]{AviMat:Laguerre}}]\label{thm:functional-Laguerre-simplified}
Let $\alpha,\beta,X,Y$ be indeterminates over $\F_p$.
There exist rational expressions $c_i(\alpha,\beta)\in\F_p(\alpha,\beta)$, such that
\[
L_{p-1}^{(\alpha)}(X)\cdot
L_{p-1}^{(\beta)}(Y)
\equiv
L_{p-1}^{(\alpha+\beta)}(X+Y)\cdot
\Bigl(
c_0(\alpha,\beta)+\sum_{i=1}^{p-1}c_i(\alpha,\beta)X^iY^{p-i}
\Bigr)
\]
in $\F_p(\alpha,\beta)[X,Y]$, modulo the ideal generated by
$X^p-(\alpha^p-\alpha)$
and
$Y^p-(\beta^p-\beta)$.
\end{theorem}

To be precise, \cite[Proposition~2]{AviMat:Laguerre} is actually stronger than Theorem~\ref{thm:functional-Laguerre-simplified},
and its formulation slightly more complicated, as it provides finer control
over the rational expressions $c_i(\alpha,\beta)$, which is needed in some applications.
We will return to this subtlety in Subsection~\ref{subsec:more}.

The crucial property of $L_{p-1}^{(\alpha)}(X)$ which allows it to play a role of an exponential in the grading switching described in~\cite{AviMat:Laguerre}
is that the factor in parentheses at the right-hand side, as well as the moduli, have only terms of total degree a multiple of $p$.

\subsection{Inverting $L_{p-1}^{(\alpha)}(X)$}\label{subsec:inverting}
The main goal of this paper is producing a left compositional inverse, and then a corresponding right inverse, of the polynomial
$L_{p-1}^{(\alpha)}(X)$ in an appropriate sense.
Because of Theorem~\ref{thm:functional-Laguerre-simplified}, in the motivating applications to algebra gradings
the Laguerre polynomial $L_{p-1}^{(\alpha)}(X)$ is of interest only regarded modulo $X^p-(\alpha^p-\alpha)$,
this dictates the context of the desired inverse.

We quickly dispose of the case of characteristic $p=2$, where $L_{1}^{(\alpha)}(X)=(1+\alpha)+X$
and hence
$L_{1}^{(\alpha)}\bigl(L_{1}^{(\alpha)}(X)\bigr)=X$
in $\F_{2}[\alpha,X]$.
Hence $L_{1}^{(\alpha)}(X)$ is its own inverse in this case, and in a strong sense.
For the rest of the paper we make the blanket assumption that $p$ is odd.

To describe the coefficients of the desired inverse of $L_{p-1}^{(\alpha)}(X)$, as a polynomial in $X$,
we introduce a family of polynomials $b_{r,s}(\alpha)\in\F_{p}[\alpha]$ depending on integers $0<r,s<p$:
\begin{equation}\label{eq:b(ra,sa)}
b_{r,s}(\alpha):=
\sum_{k=0}^{p-1}(-r/s)^k\binom{r\alpha -1}{p-1-k}\binom{s \alpha-1}{k}.
\end{equation}
Because $\binom{-1}{k}=(-1)^k$ for $k\ge 0$, for $r+s\neq p$ the polynomials
$b_{r,s}(\alpha)$ have constant term
$b_{r,s}(0)=\sum_{k=0}^{p-1}(-1)^{k} (r/s)^k=1$ (in $\F_p$).
More generally, for any positive integer $a$
the expression $b_{r,s}(a)$ equals the coefficient of $X^{p-1}$ in the polynomial
$(1+X/r)^{ra-1}(1-X/s)^{sa-1}\in\F_p[X]$,
having noted that $r^{p-1}\equiv 1\pmod{p}$.
This fact will be used in the proof of Theorem~\ref{thm:roots} to prove certain properties of the polynomials $b_{r,s}(\alpha)$.
For now it shows that $b_{r,p-r}(\alpha)\in\F_p[\alpha]$ is the zero polynomial, as it has degree at most $p-1$ by definition, but
$b_{r,p-r}(a)=0$ for any positive integer $a$
because
$(1+X/r)^{pa-2}=(1+X/r)^{p-2}\cdot(1+X^p/r)^{a-1}\in\F_p[X]$
has no term of degree $p-1$.

We will conveniently allow ourselves to interpret the integers $r$ and $s$ as elements of $\F_p^\ast$,
and write $b_{r,-r}(\alpha)$ instead of  $b_{r,p-r}(\alpha)$, for example.
With this convention, for $t\in\F_p^\ast$ we plainly have
$b_{rt,st}(\alpha)=b_{r,s}(t\alpha)$.
This allows us to assume $r=1$ in studying such polynomials and, in fact, only the polynomials $b_{1,s}(\alpha)$
are required in stating the following result.

The following result provides the desired left inverse of $L_{p-1}^{(\alpha)}(X)$.

\begin{theorem}\label{thm:inverse}
Let $\F_{p}(\alpha)$  be the field of the rational expressions in the indeterminate $\alpha$.
There is a unique  polynomial $G^{(\alpha)}(X)$ of degree less than $p$ in $\F_{p}(\alpha)[X]$ such that
\[
G^{(\alpha)}\bigl(L_{p-1}^{(\alpha)}(X)\bigr)\equiv X \pmod{X^p-(\alpha^p-\alpha)},
\]
and is given by
\[
G^{(\alpha)}(X)=-\sum_{k=1}^{p-1}\frac{1}{k} \frac{X^{k}}{\prod_{s=1}^{k-1}b_{1,s}(\alpha)}.
\]
\end{theorem}

The denominator $\prod_{s=1}^{k-1}b_{1,s}(\alpha)$ in the expression for $G^{(\alpha)}(X)$ should be read as $1$ when $k=1$.

We now describe in which sense Theorem~\ref{thm:inverse} generalizes a truncated version modulo $X^p$ of the equation
$\log\bigl(\exp(X)\bigr)=X$ in $\Q[[X]]$.
Setting $\alpha=0$ we have $L_{p-1}^{(0)}(X)=E(X)=\sum_{k=0}^{p-1}X^k/k!$, the truncated exponential in characteristic $p$, and
$G^{(0)}(X)=-\sum_{k=1}^{p-1}X^k/k=-\pounds_1(X)$.
Here $\pounds_d(X)=\sum_{k=1}^{p-1}X^k/k^d$ denote the {\em finite polylogarithms}.
They are truncated versions of the power series
$\Li_d(X)=\sum_{k=1}^{\infty}X^k/k^d$,
which serve to define the ordinary {\em polylogarithms} over the complex numbers in a neighbourhood of zero.
In particular,
$\Li_1(X)=-\log(1-X)$
is closely related to the ordinary logarithmic series.

For composition of formal power series to make sense the absence of a constant term is required,
and hence the equation $\log\bigl(\exp(X)\bigr)=X$ in $\Q[[X]]$
should really be interpreted as
$\log\bigl(1+(\exp(X)-1)\bigr)=X$.
Adopting polylogarithmic notation this reads
$-\Li_1\bigl(1-\exp(X))\bigr)=X$.
Viewing this equation modulo $X^p$, and then modulo $p$, it implies
$-\pounds_1\bigl(1-E(X))\bigr)\equiv X\pmod{X^p}$ in $\F_p[[X]]$, or actually in the polynomial ring $\F_p[X]$.
However, setting $\alpha=0$ in the congruence of Theorem~\ref{thm:inverse} yields
$-\pounds_1\bigl(E(X))\bigr)\equiv X\pmod{X^p}$ in $\F_p[X]$.
This apparent discrepancy is resolved by the functional equation
$\pounds_1(1-X)=\pounds_1(X)$, which the polynomial $\pounds_1(X)$ satisfies
(see Subsection~\ref{subsec:functional} below for more details).
Hence Theorem~\ref{thm:inverse} generalizes a truncated version
of the familiar fact that the logarithm is a left inverse of the exponential function,
modulo an application of the mentioned functional equation for $\pounds_1(X)$.

The left compositional inverse $G^{(\alpha)}(X)$ for $L_{p-1}^{(\alpha)}(X)$ given in Theorem~\ref{thm:inverse} gives rise to a right inverse as follows.

\begin{cor}\label{cor:inverse}
We have
\[
L_{p-1}^{(\alpha)}\bigl(
G^{(\alpha)}(X)
\bigr)
\equiv X\pmod{X^p-L_{p-1}^{(\alpha^p)}(\alpha^p-\alpha)}
\]
in the polynomial ring $\F_{p}(\alpha)[X]$.
\end{cor}

According to~\cite[Lemma~1]{AviMat:Laguerre}, the expression $L_{p-1}^{(\alpha^p)}(\alpha^p-\alpha)$
which appears in the modulus of the congruence in Corollary~\ref{cor:inverse} can be explicitly factorized as
\begin{equation}\label{eq:L^p}
L_{p-1}^{(\alpha^p)}(\alpha^p-\alpha)=\prod_{k=1}^{p-1}(1+\alpha/k)^k.
\end{equation}

\subsection{Additional information on the coefficients of $G^{(\alpha)}(X)$}\label{subsec:more}

As mentioned earlier, Theorem~\ref{thm:functional-Laguerre-simplified} is a weaker and simplified version of~\cite[Proposition~2]{AviMat:Laguerre}.
This is sufficient for an application to grading switching only under special circumstances.
For full generality the applications need the following additional information on the rational functions $c_i(\alpha,\beta)$.

\begin{theorem}[{\cite[Proposition~2]{AviMat:Laguerre}}]\label{thm:functional-Laguerre-more}
The rational expressions $c_i(\alpha,\beta)$ of Theorem~\ref{thm:functional-Laguerre-simplified}
belong to the subring
$\F_p\bigl[\alpha,\beta,\bigl((\alpha+\beta)^{p-1}-1\bigr)^{-1}\bigr]$ of
$\F_p(\alpha,\beta)$.
\end{theorem}

The relevance of Theorem~\ref{thm:functional-Laguerre-more} in the context of grading switching is to guarantee that
$c_i(\alpha,\beta)$ can be evaluated on elements $\tilde\alpha,\tilde\beta$ of some extension of $\F_p$
as long as $\tilde\alpha+\tilde\beta\not\in\F_p^\ast$.
For example, when $p=3$ the rational expressions $c_0(\alpha,\beta)$, $c_1(\alpha,\beta)$, and $c_2(\alpha,\beta)$ equal, respectively,
\[
\frac{(1-\alpha^2)(1-\beta^2)}{1-(\alpha+\beta)^2},
\qquad
\frac{\alpha-1}{1-(\alpha+\beta)^2},
\qquad
\frac{\beta-1}{1-(\alpha+\beta)^2}.
\]
We supplement our Theorem~\ref{thm:inverse} with corresponding information on the coefficients of $G^{(\alpha)}(X)$.

\begin{theorem}\label{thm:inverse-more}
The coefficients of the polynomial $G^{(\alpha)}(X)$ of Theorem~\ref{thm:inverse} belong to the subring
$\F_{p}[\alpha, (\alpha^{p-1}-1)^{-1}]$ of $\F_p(\alpha)$.
\end{theorem}


As an example, when $p=3$ we have
$G^{(\alpha)}(X) =-X -X^2/(\alpha+2)$.
Theorem~\ref{thm:inverse-more} will follow at once from the identity
$b_{1,s}(\alpha)b_{1,s}(-\alpha)=1-\alpha^{p-1}$
for $s=1, \dots, p-2$, which we will prove in Lemma~\ref{lemma:b(a)b(-a)}
of Section~\ref{sec:roots}.
In particular, this identity shows that $b_{1,s}(\alpha)$ is a polynomial of degree $(p-1)/2$, and its roots are exactly half the elements of $\F_p^{\ast}$.

We devote the remainder of Section~\ref{sec:roots} to finding a simple way to decide exactly which elements of $\F_p^{\ast}$ are roots of $b_{1,s}(\alpha)$,
in Theorem~\ref{thm:roots}.
That result provides a simpler description of the polynomials $b_{1,s}(\alpha)$ and, consequently, of the coefficients of $G^{(\alpha)}(X)$.

\subsection{Functional equations for the polynomials $G^{(\alpha)}(X)$}\label{subsec:functional}
The interpretation of the polynomials $G^{(\alpha)}(X)$ as generalized logarithms suggests that they might satisfy
some analogue of the functional equation $\log(xy)=\log(x)+\log(y)$.
After all, according to Theorem~\ref{thm:functional-Laguerre-simplified} the generalized exponentials $L_{p-1}^{(\alpha)}(X)$,
which the $G^{(\alpha)}(X)$ suitably invert,
satisfy a weak version of the functional equation $\exp(x)\exp(y)=\exp(x+y)$.
Unfortunately, switching from a weak functional equation for a weak exponential to a corresponding functional equation for a weak logarithm
does not seem to go through as smoothly as it would for their classical analogues.
The picture is enriched, but also made much more complex, by the presence of additional functional equations
for the truncated logarithm, which we briefly discuss now.

As we observed after Theorem~\ref{thm:inverse}, the polynomials $G^{(\alpha)}(X)$ are a parametrized version of the finite polylogarithm
$\pounds_1(X)=\sum_{k=1}^{p-1}X^k/k$,
because $G^{(0)}(X)=-\pounds_1(X)$.
The polynomial $\pounds_1(X)$ (viewed in characteristic $p$ as we do throughout this paper)
satisfies certain functional equations which have no analogues for the classical logarithm which it resembles.
One such equation is
$\pounds_1(X)=-X^p\pounds_1(1/X)$,
which is an immediate consequence of Wilson's theorem, $(p-1)!=-1$ in $\F_p$.
This functional equation does extend to the following equation for the polynomials $G^{(\alpha)}(X)$.

\begin{theorem}\label{thm:reciprocal}
The polynomials $G^{(\alpha)}(X)$ satisfy
\[
\prod_{k=1}^{p-1}(1+\alpha/k)^{k}\cdot
G^{(\alpha)}(X)=
-X^{p}\cdot
G^{(-\alpha)}\left(\frac{1-\alpha^{p-1}}{X}\right)
\]
in the polynomial ring $\F_{p}(\alpha)[X]$.
\end{theorem}

We will prove Theorem~\ref{thm:reciprocal} in Section~\ref{sec:proofs_B}, using the explicit characterization and properties of their coefficients $b_{1,s}(\alpha)$
obtained in Section~\ref{sec:roots}.

Another functional equation satisfied by $\pounds_1(X)$ is $\pounds_1(1-X)=\pounds_1(X)$.
See~\cite[Equation~(13)]{MatTau:polylog} or~\cite[Equation~(28)]{MatTau:truncation} for the easy proof,
and~\cite[Lemma~3.2]{MatTau:polylog} for a generalization, probably already known to Mirimanoff at the beginning of the twentieth century~\cite{Mirimanoff}.
We do not know whether this functional equation extends to an equation involving polynomials $G^{(\alpha)}(X)$ (in a sensible way), and we formulate this as a question.

\begin{question}
Do the polynomials $G^{(\alpha)}(X)$, perhaps collectively, satisfy a generalization of the functional equation $\pounds_1(1-X)=\pounds_1(X)$
for $\pounds_1(X)=-G^{(0)}(X)$?
\end{question}

Alternate application of the functional equations
$\pounds_1(X)=\pounds_1(1-X)$
and
$\pounds_1(X)=-X^p\pounds_1(1/X)$
generates six equivalent expressions for $\pounds_1(X)$ (as a polynomial in $\F_p[X]$), namely,
\begin{multline*}
\pounds_1(X)
=\pounds_1(1-X)
=(X-1)^p\pounds_1\left(\frac{1}{1-X}\right)
=(X-1)^p\pounds_1\left(\frac{X}{X-1}\right)
\\
=-X^p\pounds_1\left(\frac{X-1}{X}\right)
=-X^p\pounds_1\left(\frac{1}{X}\right).
\end{multline*}
This invariance of $\pounds_1(X)$ under a certain group of {\em symmetries} of order six
accounts for multiple representations of any formula involving $\pounds_1(X)$.
See~\cite{MatTau:polylog} for further discussion and applications of those symmetries.
Of course generalizations to $G^{(\alpha)}(X)$ remain inaccessible as long as we miss a generalization of $\pounds_1(X)=\pounds_1(1-X)$.

The following result, also proved in Section~\ref{sec:proofs_B}, is an analogue of the equation
$\log(x^h)=h\log(x)$.

\begin{theorem}\label{thm:powers}
For any integer $h$ with $0<h<p$ we have
\[
G^{(h\alpha)}\left(\frac{X^h}{\prod_{s=1}^{h-1}b_{1,s}(\alpha)}\right)
\equiv hG^{(\alpha)}(X)
\pmod{X^p-L_{p-1}^{(\alpha^p)}(\alpha^p-\alpha)}
\]
in the polynomial ring $\F_{p}(\alpha)[X]$.
\end{theorem}

According to Lemma~\ref{lemma:product},
when $h=p-1$ the equation of Theorem~\ref{thm:powers} can also be written as
\[
G^{(-\alpha)}\left(\frac{X^{p-1}\cdot(1-\alpha^{p-1})}{L_{p-1}^{(\alpha^p)}(\alpha^p-\alpha)}\right)
\equiv -G^{(\alpha)}(X)
\pmod{X^p-L_{p-1}^{(\alpha^p)}(\alpha^p-\alpha)}.
\]
However, this last congruence is also a weak consequence of Theorem~\ref{thm:reciprocal}.

Of course $\log(x^h)=h\log(x)$ for integer $h$ is a consequence of the fundamental functional equation $\log(xy)=\log(x)+\log(y)$.
That equation has an analogue for $\pounds_1(X)$ in the polynomial ring $\F_p[X,Y]$, namely,
\begin{equation}\label{eq:4-term}
\pounds_1(X)-\pounds_1(Y)
+X^p\pounds_1\left(\frac{Y}{X}\right)
+(1-X)^p\pounds_1\left(\frac{1-Y}{1-X}\right)
=0.
\end{equation}
This equation seems to have been first noticed (at least in recent times) by Kontsevich~\cite{Kontsevich}, in a somehow more rudimentary form,
see~\cite[Equation~(31)]{MatTau:truncation} and the subsequent discussion.
It admits several proofs, two of which are given in~\cite{MatTau:truncation}.
We sketch a variation of one of those proofs which emphasizes its origin in the fundamental functional equation of the logarithm,
in its variant
\[
-\log(1-X)+\log(1-Y)
-\log\left(\frac{1-Y}{1-X}\right)
=0,
\]
which takes place in the power series ring $\Q[[X,Y]]$.
Viewing modulo the ideal $(X,Y)^p$ and then reducing modulo $p$ one finds
\[
\pounds_1(X)-\pounds_1(Y)
+\pounds_1\left(1-\frac{1-Y}{1-X}\right)
\equiv 0\pmod{(X,Y)^p}
\]
in the power series ring $\F_p[[X,Y]]$.
Rewriting the last summand using $\pounds_1(1-Z)=\pounds_1(Z)$ we find that a weak version of Equation~\eqref{eq:4-term}
holds as a congruence modulo the ideal $(X,Y)^p$ of $\F_p[[X,Y]]$.
The symmetries (or functional equations) for $\pounds_1(X)$ described earlier allow one to, both recover the missing term
in Equation~\eqref{eq:4-term}, and prove that it holds in full, first in $\F_p[[X,Y]]$ but then in $\F_p[X,Y]$ where Equation~\eqref{eq:4-term} takes place,
see~\cite[Equation~(31)]{MatTau:truncation} for the details.
Unfortunately, we are not able to extend these arguments to a generalization for $G^{(\alpha)}(X)$,
and so we leave that as our final question.

\begin{question}
Do the polynomials $G^{(\alpha)}(X)$ satisfy a generalization of Equation~\eqref{eq:4-term}
for $\pounds_1(X)=-G^{(0)}(X)$?
\end{question}

Such a generalization would have to be a congruence in the guise of the equation in Theorem~\ref{thm:functional-Laguerre-simplified},
involving $G^{(\alpha)}(X)$, $G^{(\beta)}(Y)$, and other terms.

\section{Proofs of Theorem~\ref{thm:inverse} and Corollary~\ref{cor:inverse}}\label{sec:proofs}

The first step towards computing $G^{(\alpha)}\bigl(L_{p-1}^{(\alpha)}(X)\bigr)$
modulo $X^p-(\alpha^p-\alpha)$ as in Theorem~\ref{thm:inverse}
is computing the powers of $L_{p-1}^{(\alpha)}(X)$ modulo $X^p-(\alpha^p-\alpha)$.
For that we need to consider the products $L_{p-1}^{(r\alpha)}(rX) L_{p-1}^{(s\alpha)}(sX)$, for $0<r,s<p$.
Note that under the simultaneous substitutions $X^\prime=rX$ and $\alpha^\prime=r\alpha$,
for some $0<r<p$, the modulus $X^p-(\alpha^p-\alpha)$ gets simply multiplied by the nonzero scalar $r$.

Thus, as a special case of Theorem~\ref{thm:functional-Laguerre-simplified},
after replacing $X$ with $rX$, $Y$ with $sX$, and similar substitutions for the parameters $\alpha$ and $\beta$, we find
\begin{equation*}
L_{p-1}^{(r\alpha)}(rX)
L_{p-1}^{(s\alpha)}(sX)
\equiv
b'_{r,s}(\alpha)\cdot
L_{p-1}^{((r+s)\alpha)}\bigl((r+s)X\bigr)
\pmod{X^p-(\alpha^p-\alpha)},
\end{equation*}
where
\begin{equation}\label{eq:b'(ra,sa)}
b'_{r,s}(\alpha)=
c_0(r\alpha,s\alpha)+(\alpha^{p}-\alpha)\sum_{i=1}^{p-1}c_i(r\alpha,s \alpha)r^{i}s^{p-i},
\end{equation}
and $c_i(r\alpha,s\alpha)$ are certain rational expressions in $\F_{p}(\alpha)$.
We now prove that the rational expressions $b'_{r,s}(\alpha)$ are actually polynomials, and coincide with the polynomials
$b_{r,s}(\alpha)$ defined in Equation~\eqref{eq:b(ra,sa)} as long as $r+s\neq p$.

\begin{lemma}\label{lemma:product}
For integers $0<r,s<p$ with $r+s\neq p$ we have
\[
L_{p-1}^{(r\alpha)}(rX)\cdot
L_{p-1}^{(s\alpha)}(sX)
\equiv b_{r,s}(\alpha)\cdot L_{p-1}^{((r+s)\alpha)}\bigl((r+s)X\bigr)
\pmod{X^p-(\alpha^p-\alpha)},
\]
in $\F_{p}(\alpha)[X]$, where $b_{r,s}(\alpha)\in\F_{p}[\alpha]$ are the polynomials defined in Equation~\eqref{eq:b(ra,sa)}.
Furthermore, in the excluded case where $r+s=p$ we have
\[
L_{p-1}^{(r\alpha)}(rX)\cdot
L_{p-1}^{(-r\alpha)}(-rX)
\equiv 1-\alpha^{p-1}
\pmod{X^p-(\alpha^p-\alpha)}.
\]
\end{lemma}

\begin{proof}
Expanding the product
$L_{p-1}^{(r\alpha)}(rX) \cdot L_{p-1}^{(s\alpha)}(sX)$
according to Theorem~\ref{thm:functional-Laguerre-simplified}
we have
\[
L_{p-1}^{(r\alpha)}(rX) \cdot L_{p-1}^{(s\alpha)}(sX)
=
\sum_{k=0}^{p-1}\sum_{h=0}^{p-1}\binom{r \alpha-1}{p-1-k}\binom{s\alpha-1}{p-1-h}
\frac{(-r)^k(-s)^h}{k!h!}X^{k+h}
\]
Separating the terms in the double sum according to whether the degree $k+h$ is less than $p$ or at least $p$, the sum of the former terms equals
\[
\sum_{t=0}^{p-1}\sum_{k=0}^{t}\binom{r \alpha-1}{p-1-k}\binom{s\alpha-1}{p-1-t+k}
\frac{(-1)^{t}r^{k}s^{t-k}}{k!(t-k)!}X^{t},
\]
having set $k+h=t$,
and the sum of the latter is congruent to
\[
-s(\alpha^p-\alpha)\sum_{t=0}^{p-2}\sum_{k=t+1}^{p-1}\binom{r \alpha-1}{p-1-k}
\binom{s\alpha-1}{k-t-1}
\frac{(-1)^{t}r^{k}s^{t-k}}{k!(p+t-k)!}X^{t}
\]
modulo $X^p-(\alpha^p-\alpha)$, having set $k+h=p+t$ and
taken $s^p\equiv s\pmod{p}$ into account.

Note that both sums are polynomials of degree less than $p$.
Hence the rational expression $b'_{r,s}(\alpha)$ of Equation~\ref{eq:b'(ra,sa)}
can now be found by equating the
coefficients of any suitable power of $X$ in
\[
b'_{r,s}(\alpha)\cdot L_{p-1}^{((r+s)\alpha)}\bigl((r+s)X\bigr)
=
b'_{r,s}(\alpha)\cdot\sum_{k=0}^{p-1}\binom{(r+s)\alpha-1}{p-1-k}\frac{\bigl(-(r+s)X\bigr)^k}{k!},
\]
and in the expression for $L_{p-1}^{(r\alpha)}(rX) \cdot L_{p-1}^{(s\alpha)}(sX)$ modulo $X^p-(\alpha^p-\alpha)$ found above.

Assuming $r+s\neq 0$ first, the coefficient of $X^{p-1}$ in the former equals $-b'_{r,s}(\alpha)$, and in the latter it equals
\[
\sum_{k=0}^{p-1}\binom{r\alpha-1}{p-1-k}\binom{s\alpha-1}{k}
\frac{r^{k}s^{p-1-k}}{k!(p-1-k)!}.
\]
Because $s^{p-1}\equiv 1\pmod{p}$ and
$1/\bigl(k!(p-1-k)!\bigr)\equiv -\binom{p-1}{k}\equiv -(-1)^k\pmod{p}$,
we conclude that
$b'_{r,s}(\alpha)=b_{r,s}(\alpha)$ in this case.

When $r+s=p$ the argument needs to be modified because
$L_{p-1}^{((r+s)\alpha)}\bigl((r+s)X\bigr)=L_{p-1}^{(0)}(0)=1$
has no term of degree $p-1$.
Hence we compare constant terms instead, which is a little more complicated.
The constant term in the reduced expression for
$L_{p-1}^{(r\alpha)}(rX)\cdot L_{p-1}^{(-r\alpha)}(-rX)$ found above equals
\[
\binom{r\alpha-1}{p-1}\binom{-r\alpha-1}{p-1}
+
r(\alpha^p-\alpha)\sum_{k=1}^{p-1}\binom{r\alpha-1}{p-1-k}
\binom{-r\alpha-1}{k-1}
\frac{(-1)^k}{k!(p-k)!}.
\]
To deal with the first isolated summand note that
$\binom{-r\alpha-1}{p-1}=1-\alpha^{p-1}$ in $\F_p[\alpha]$,
for example because its roots are $1/r,2/r,\ldots,(p-1)/r$, hence exactly all nonzero elements of $\F_p$,
and its leading term equals $(-r\alpha)^{p-1}/(p-1)!=\alpha^{p-1}$
(recalling that $p$ is odd throughout the paper).
The sum over $k$ can be transformed using
the congruence $k!(p-k)!\equiv (-1)^{k}k\pmod{p}$ for $0<k<p$,
and the standard binomial identity $\binom{n}{k}=\frac{n}{k}\binom{n-1}{k-1}$.
Altogether, the constant term in the reduced expression for
$L_{p-1}^{(r\alpha)}(rX)\cdot L_{p-1}^{(-r\alpha)}(-rX)$
found earlier can be written as
\[
(1-\alpha^{p-1})\sum_{k=0}^{p-1}\binom{r\alpha-1}{p-1-k}\binom{-r\alpha}{k}
=(1-\alpha^{p-1})\binom{-1}{p-1}
=1-\alpha^{p-1}.
\]
Comparing this with
$b'_{r,p-r}(\alpha)\cdot L_{p-1}^{(0)}(0)=b'_{r,p-r}(\alpha)$
at the other side, we find
$b'_{r,p-r}(\alpha)=1-\alpha^{p-1}$
as desired.
\end{proof}

An inductive application of Lemma~\ref{lemma:product} now yields
\begin{equation}\label{eq:power}
\bigl(L_{p-1}^{(\alpha)}(X)\bigr)^{j}
\equiv
\Bigl(\prod_{s=1}^{j-1} b_{1,s}(\alpha)\Bigr)
\cdot
L_{p-1}^{(j\alpha)}(jX)
\pmod{X^p-(\alpha^p-\alpha)},
\end{equation}
for $0<j<p$,
where the product over $s=1,\ldots,j-1$ is interpreted as $1$ when $j=1$.
After these preparations we are ready to prove Theorem~\ref{thm:inverse}.

\begin{proof}[Proof of Theorem~\ref{thm:inverse}]
Let $G^{(\alpha)}(X)$ be an arbitrary polynomial
of degree less than $p$ in $\F_{p}(\alpha)[X]$ satisfying
\begin{equation}\label{eq:GL}
G^{(\alpha)}\bigl(L_{p-1}^{(\alpha)}(X)\bigr)\equiv X \pmod{X^p-(\alpha^p-\alpha)}.
\end{equation}
From the specialization $\alpha=0$, where $L_{p-1}^{(0)}(X)=E(X)$ is the truncated exponential and hence
$-\pounds_1\bigl(L_{p-1}^{(0)}(X)\bigr)\equiv X \pmod{X^p}$
with
$\pounds_1(X)=\sum_{k=1}^{p-1}X^k/k$,
we already know that $G^{(0)}(X)=-\pounds_1(X)$
(see the discussion after Theorem~\ref{thm:inverse}.)
From this fact together with Equation~\eqref{eq:power} we expect a cleaner calculation if we write
$G^{(\alpha)}(X)$ directly in the form
\[
G^{(\alpha)}(X)=
c_0(\alpha)
-\sum_{j=1}^{p-1}\frac{c_j(\alpha)}{j} \frac{X^{j}}{\prod_{s=1}^{j-1}b_{1,s}(\alpha)},
\]
where $c_j(\alpha)\in\F_p(\alpha)$ are rational functions to be determined.
According to Equation~\eqref{eq:power} then Equation~\eqref{eq:GL} is equivalent to
\[
c_0(\alpha)
-\sum_{j=1}^{p-1}\frac{c_j(\alpha)}{j}
L_{p-1}^{(j\alpha)}(jX)
\equiv X
\pmod{X^p-(\alpha^p-\alpha)}.
\]

Because both sides are polynomials of degree less than $p$, the congruence must be an equality.
Equating the coefficients of like powers of $X$ on both sides, and because
\[
L_{p-1}^{(j\alpha)}(jX)
=-\sum_{k=0}^{p-1}(j\alpha-1)_{p-1-k}\cdot (jX)^k
\]
we find the system of linear equations
\begin{equation}\label{eq:system}
\sum_{j=1}^{p-1}
(j\alpha-1)_{p-1-k}\cdot j^{k-1}
\cdot c_j(\alpha)
=
\begin{cases}
-c_0(\alpha)&\text{if $k=0$}\\
1&\text{if $k=1$}\\
0&\text{if $1<k<p$}
\end{cases}
\end{equation}
for the expressions $c_j(\alpha)$,
over the field $\F_p(\alpha)$.

Because
$(j\alpha-1)_{p-1}=(j\alpha)^{p-1}-1=\alpha^{p-1}-1$,
the first equation can be written as
$(\alpha^{p-1}-1)\sum_{j=1}^{p-1}
j^{-1}\cdot c_j(\alpha)
=-c_0(\alpha)$.
As to the remaining equations, by successively adding to each equation a suitable linear combination of the equations which follow it,
starting from the end, we show that the system is equivalent to
\begin{equation}\label{eq:system_simpler}
\sum_{j=1}^{p-1}
j^{k-1}
\cdot c_j(\alpha)
=
\begin{cases}
c_0(\alpha)/(1-\alpha^{p-1})&\text{if $k=0$}\\
-1&\text{if $k=1$}\\
0&\text{if $1<k<p$}
\end{cases}
\end{equation}
The last equation in \eqref{eq:system}, namely for $k=p-1$, reads indeed
$\sum_{j=1}^{p-1}
j^{p-2}\cdot c_j(\alpha)=0$,
which is the last equation in~\eqref{eq:system_simpler}.
Now fix $0<k<p-1$, assume we have proved the conclusion for all equations with index higher than $k$
(that is, we have already obtained the last $p-1-k$ equations in~\eqref{eq:system_simpler}),
and consider the equation with index $k$ in~\eqref{eq:system}.
If we expand the expression
$(j\alpha-1)_{p-1-k}$
in its left-hand side as a polynomial in $\alpha$,
each term except for the constant term gives rise to a scalar multiple (with scalar in $\F_p(\alpha)$)
of the left-hand side of an equation in~\eqref{eq:system_simpler} with index larger than $k$.
The constant term in $(j\alpha-1)_{p-1-k}$
gives rise precisely to the left-hand side of the equation with index $k$ in~\eqref{eq:system_simpler},
multiplied by $(-1)^k(p-1-k)!$.

It is now an easy matter to see that the system~\eqref{eq:system_simpler} has the unique solution
$c_0(\alpha)=0$, and $c_j(\alpha)=1$ for $0<j<p$, as desired,
using the fact that
$\sum_{j=1}^{p-1}j^h$ is congruent to $-1$ modulo $p$ if $p-1$ divides $h$, and to $0$ otherwise.
One may also view $j$ as ranging over the multiplicative group $\F_p^\ast$,
interpret equations~\eqref{eq:system_simpler} as performing a discrete Fourier transform over $\F_p^\ast$,
and appropriately invert it.
\end{proof}

Now we prove that the left inverse $G^{(\alpha)}(X)$ for $L_{p-1}^{(\alpha)}(X)$ is also a right inverse, in the appropriate sense
stated in Corollary~\ref{cor:inverse}.

\begin{proof}[Proof of Corollary~\ref{cor:inverse}]
Some care is needed to transfer the problem to a setting where we can use associativity of composition
to pass from a unique left inverse to a right inverse, and a convenient setting is that of a ring of formal power series.

The modulus $X^p-(\alpha^p-\alpha)$ in the congruence of Theorem~\ref{thm:inverse} becomes a $p$th power
once we extend the field of coefficients to $\F_p(\alpha^{1/p})$.
Viewing $L_{p-1}^{(\alpha)}(X)$ and $G^{(\alpha)}(X)$
as polynomials in $\F_p(\alpha^{1/p})[X]$,
and replacing $X$ with $X+\alpha-\alpha^{1/p}$, the congruence of Theorem~\ref{thm:inverse}
is equivalent to
\[
G^{(\alpha)}\bigl(L_{p-1}^{(\alpha)}(X+\alpha-\alpha^{1/p})\bigr)-(\alpha-\alpha^{1/p})\equiv X \pmod{X^p}
\]
in the polynomial ring $\F_p(\alpha^{1/p})[X]$.

Set $\delta=L_{p-1}^{(\alpha)}(\alpha-\alpha^{1/p})\in\F_p(\alpha^{1/p})$.
The polynomial $L_{p-1}^{(\alpha)}(X+\alpha-\alpha^{1/p})-\delta$ has no constant term, and nonzero term of degree one.
Hence it has a compositional (bilateral) inverse $S(X)$ in the power series ring $\F_p(\alpha^{1/p})[[X]]$, meaning that
\begin{equation}\label{eq:left_inverse}
S\bigl(L_{p-1}^{(\alpha)}(X+\alpha-\alpha^{1/p})-\delta\bigr)=X
\end{equation}
and
\begin{equation}\label{eq:right_inverse}
L_{p-1}^{(\alpha)}\bigl(S(X)+\alpha-\alpha^{1/p}\bigr)-\delta=X
\end{equation}
in $\F_p(\alpha^{1/p})[[X]]$.
In particular, Equation~\eqref{eq:left_inverse} yields
\[
S\bigl(L_{p-1}^{(\alpha)}(X+\alpha-\alpha^{1/p})-\delta\bigr)\equiv X \pmod{X^p}
\]
in $\F_p(\alpha^{1/p})[[X]]$,
and because this congruence alone determines $S(X)$ modulo $X^p$ uniquely, we deduce
\begin{equation}\label{eq:S_equiv_G}
S(X)\equiv G^{(\alpha)}(X+\delta)-(\alpha-\alpha^{1/p})\pmod{X^p}.
\end{equation}

Viewing Equation~\eqref{eq:right_inverse} modulo $X^p$ and then taking Equation~\eqref{eq:S_equiv_G} into account we find
\[
L_{p-1}^{(\alpha)}\bigl(
G^{(\alpha)}(X+\delta)
\bigr)-\delta
\equiv X\pmod{X^p},
\]
in $\F_p(\alpha^{1/p})[[X]]$, but because this congruence involves only polynomials it already takes place in
in $\F_p(\alpha^{1/p})[X]$.
Substituting $X-\delta$ for $X$ we find
\[
L_{p-1}^{(\alpha)}\bigl(
G^{(\alpha)}(X)
\bigr)
\equiv X\pmod{X^p-L_{p-1}^{(\alpha^p)}(\alpha^p-\alpha)},
\]
and this takes place and holds in $\F_p(\alpha)[X]$, as desired.
\end{proof}

\section{Factorizations of the polynomials $b_{1,s}(\alpha)$}\label{sec:roots}

In this section we compute the full factorizations in $\F_p[\alpha]$ of the polynomials $b_{1,s}(\alpha)$,
which will give a more explicit description of the generalized truncated logarithm $G^{(\alpha)}(X)$.
The first step is establishing that all roots in a splitting field are simple and belong to the prime field $\F_p$.
One way to show that is by proving the identity of Lemma~\ref{lemma:b(a)b(-a)} below,
which we will do by suitable manipulations of products of Laguerre polynomials.

\begin{lemma}\label{lemma:b(a)b(-a)}
The polynomials $b_{1,s}(\alpha)$, for $0<s<p-1$, satisfy
\begin{equation*}
b_{1,s}(\alpha)\cdot b_{1,s}(-\alpha)=1-\alpha^{p-1}.
\end{equation*}
In particular,
$b_{1,s}(\alpha)$ has degree $(p-1)/2$,
and factorizes into a product of distinct linear factors in $\F_{p}[\alpha]$.
\end{lemma}

\begin{proof}
We prove the claimed equation by induction on $s$,
using the congruences of Lemma~\ref{lemma:product} and omitting the modulus $X^{p}-(\alpha^{p}-\alpha)$ for conciseness.
The case $s=1$ follows by computing
\begin{align*}
(1-\alpha^{p-1})^2
&\equiv
\bigl(L_{p-1}^{(\alpha)}(X)\cdot L_{p-1}^{(-\alpha)}(-X)\bigr)^2
\\&=
L_{p-1}^{(\alpha)}(X)^2\cdot L_{p-1}^{(-\alpha)}(-X)^2
\\&\equiv
b_{1,1}(\alpha)\cdot L_{p-1}^{(2\alpha)}(2X)
\cdot b_{1,1}(-\alpha)\cdot L_{p-1}^{(-2\alpha)}(-2X)
\\&=
b_{1,1}(\alpha)\cdot b_{1,1}(-\alpha)\cdot (1-\alpha^{p-1}).
\end{align*}
Now let $1<s<p-1$ and assume we have proved
$b_{1,t}(\alpha)\cdot b_{1,t}(-\alpha)=1-\alpha^{p-1}$
for $0<t<s$. Then the induction step follows by computing
\begin{align*}
(1-\alpha^{p-1})^{s+1}
&\equiv
\bigl(L_{p-1}^{(\alpha)}(X)\cdot L_{p-1}^{(-\alpha)}(-X)\bigr)^{s+1}
\\&=
L_{p-1}^{(\alpha)}(X)^{s+1}\cdot L_{p-1}^{(-\alpha)}(-X)^{s+1}
\\&\equiv
\Bigl(\prod_{t=1}^{s}b_{1,t}(\alpha)\Bigr)
\cdot L_{p-1}^{((s+1)\alpha)}((s+1)X)
\\&\quad
\cdot\Bigl(\prod_{t=1}^{s}b_{1,t}(-\alpha)\Bigr)
\cdot L_{p-1}^{(-(s+1)\alpha)}(-(s+1)X)
\\&=
b_{1,s}(\alpha)\cdot b_{1,s}(-\alpha)\cdot(1-\alpha^{p-1})^s.
\end{align*}
The statement on the factorization of $b_{1,s}(\alpha)$ follows at once.
\end{proof}

Theorem~\ref{thm:inverse-more} follows at once.
Lemma~\ref{lemma:b(a)b(-a)} also tells us that exactly one of each pair of opposite elements of $\F_p^\ast$
is a root of a given $b_{1,s}(\alpha)$.
The following result determines precisely which of them.

\begin{theorem}\label{thm:roots}
Let $0<s<p-1$, and let $0<a,a'<p$ be integers such that $a'\equiv sa\pmod{p}$.
Then $b_{1,s}(a)=0$ if and only if $a+a'<p$.
\end{theorem}

Note that we cannot have $a+a'=p$ because we have assumed $s\not\equiv -1\pmod{p}$, and so
either $a+a'<p$ or $a+a'>p$.

\begin{proof}
Suppose first that $a+a'<p$.
We have already mentioned that $b_{1,s}(a)$ equals the coefficient
of $X^{p-1}$ in the product $(1+X)^{a-1}(1-X/s)^{sa-1}\in\F_p[X]$.
Writing $sa=a'+kp$ for some integer $k\ge 0$ we may write
\[
(1+X)^{a-1}(1-X/s)^{sa-1}=
(1+X)^{a-1}(1-X/s)^{a'-1}\bigl(1-(X/s)^p\bigr)^k
\]
in $\F_p[X]$.
Hence $b_{1,s}(a)$ equals the coefficient
of $X^{p-1}$ in the product $(1+X)^{a-1}(1-X/s)^{a'-1}$,
which is zero because this polynomial has degree $a+a'-2<p-2$.

Now suppose that
$a+a'>p$.
Then $(p-a)+(p-a')<p$, applying the implication already proved to $p-a$ in place of $a$ and $p-a'$ in place of $a'$
shows $b_{1,s}(p-a)=0$,
which is the same as
$b_{1,s}(-a)=0$.
According to Lemma~\ref{lemma:b(a)b(-a)} this implies $b_{1,s}(a)\neq 0$, which is the desired conclusion.
\end{proof}

\begin{rem}\label{rem:Lucas}
Using Lucas' theorem on binomial coefficients modulo $p$, Theorem~\ref{thm:roots} admits the following equivalent formulation,
which may in some cases be more convenient to apply.
Let $0<s<p-1$ and $0<a<p$.
Then $b_{1,s}(a)=0$ if and only if $p$ does not divide the binomial coefficient $\binom{a+sa}{a}$.
\end{rem}

Theorem~\ref{thm:roots} gives an explicit way of writing out the complete factorization of each $b_{1,s}(\alpha)$.
For example, when $s=1$ in the notation of Theorem~\ref{thm:roots} we have $a'=a$, and hence the condition $a+a'<p$
reads $a\le(p-1)/2$.
Knowing that the polynomial has constant term $1$ yields
$b_{1,1}(\alpha)=\prod_{a=1}^{(p-1)/2}(1-\alpha/a)$.
The equivalent expression
$b_{1,1}(\alpha)=(-1)^{(p-1)/2}\binom{\alpha-1}{(p-1)/2}$
could also be obtained directly by recognizing an alternating sign Chu-Vandermonde convolution in the defining Equation~\eqref{eq:b(ra,sa)}.
The latter route is not an option when $s=2$, in which case Theorem~\ref{thm:roots} easily yields
\[
b_{1,2}(\alpha)=
\biggl(\prod_{0<a<p/3}(1-\alpha/a)\biggr)
\cdot
\biggl(\prod_{p/2<a<2p/3}(1-\alpha/a)\biggr).
\]

For a further example consider the case of $s=(p-1)/2$.
In the notation of Theorem~\ref{thm:roots},
when $a$ is odd we have $a'=(p-a)/2$, whence $a+a'=(p+a)/2<p$, and
when $a$ is even we have $a'=p-a/2$, whence $a+a'=p+a/2>p$.
Consequently, the roots of $b_{1,(p-1)/2}(\alpha)$ are $1,3,5,\ldots,p-2$
(viewed as elements of $\F_p$), and hence
\[
b_{1,(p-1)/2}(\alpha)
=(-1)^{\frac{p+1}{2}}\binom{(\alpha-1)/2}{(p-1)/2}.
\]

As a more substantial application of Theorem~\ref{thm:roots}, we next use it to establish a symmetry property of the polynomials $b_{1,s}(\alpha)$.

\begin{cor}\label{cor:symmetry}
For $0<s<p-1$ we have
$b_{1,s}(\alpha)=b_{1,p-1-s}(\alpha)$.
\end{cor}

\begin{proof}
Because the two polynomials have the same constant term $1$, and each has distinct roots, it is sufficient to show that those roots are the same.
Also, by symmetry reasons it is sufficient to show that if $b_{1,s}(a)=0$ for some $0<a<p$, then $b_{1,p-1-s}(a)=0$.
In fact, according to Theorem~\ref{thm:roots} the latter condition holds exactly when $a+a'<p$,
where $0<a'<p$ is uniquely determined by $a'\equiv sa\pmod{p}$.
But then $p-a-a'\equiv (p-1-s)a\pmod{p}$ and $0<p-a-a'<p$, and because $0<a+(p-a-a')=p-a'<p$
another application of Theorem~\ref{thm:roots} yields $b_{1,p-1-s}(a)=0$, as desired.
\end{proof}

\begin{rem}
With the alternate interpretation of Remark~\ref{rem:Lucas}, one can also prove Corollary~\ref{cor:symmetry} by noting that
$\binom{a+(p-1-s)a}{a}=\binom{pa-sa}{a}\equiv\binom{-sa}{a}\pmod{p}$, and then
$
\binom{-sa}{a}=
(-1)^a\binom{a+sa-1}{a}=
(-1)^a\frac{s}{s+1}\binom{a+sa}{a}.
$
\end{rem}

We can also use Theorem~\ref{thm:roots} to compute the coefficient of $X^{p-1}$ in $G^{(\alpha)}(X)$, which equals the reciprocal of the product
$\prod_{s=1}^{p-2}b_{1,s}(\alpha)$.

\begin{cor}\label{cor:product}
We have
\[
\prod_{s=1}^{p-2}b_{1,s}(\alpha)=\prod_{k=2}^{p-1}(1+\alpha/k)^{k-1}
=\frac{L_{p-1}^{(\alpha^p)}(\alpha^p-\alpha)}{1-\alpha^{p-1}}.
\]
\end{cor}

\begin{proof}
Fix $0<a<p$.
As $s$ ranges over $0<s<p-1$, the residue class of $sa$ ranges over all residue classes modulo $p$ with the exception of the classes of $0$ and of $-a$.
Consequently, the conditions $a'\equiv sa\pmod{p}$ and $0<a'<p-a$ of Theorem~\ref{thm:roots}
will be simultaneously satisfied for exactly $p-a-1$ values of $s$.
This means that $a$ will be a root, and a single root as we know, of exactly $p-a-1$ of the factors in the product $\prod_{s=1}^{p-2}b_{1,s}(\alpha)$.
This proves the first claimed equality.
The second equality follows from Equation~\eqref{eq:L^p}, which was proved in~\cite[Lemma~1]{AviMat:Laguerre}.
\end{proof}

In the proof of Theorem~\ref{thm:roots} we have used the fact that
if $a$ is positive integer then $b_{r,s}(a)$ equals the coefficient of $X^{p-1}$ in the polynomial
$(1+X/r)^{ra-1}(1-X/s)^{sa-1}\in\F_p[X]$.
We conclude this section with noting that a similar fact holds true for $b_{r,s}(\alpha)$ as a polynomial, provided that we make sense of powers with
exponents $r\alpha-1$ and $s\alpha-1$.

The obstacle here is that $(1+X/r)^{r\alpha-1}$ does not have a meaning over the field $\F_p$ if $\alpha$ is an indeterminate.
If $C$ is any commutative ring containing $\Q$ as a subring,
then for any $c\in C$ we have a binomial series defined as
\[
(1+X)^c
=\sum_{k=0}^{\infty}\binom{c}{k}X^k
\in C[[X]].
\]
This does not generally make sense in prime characteristic $p$ because the binomial coefficients
$\binom{c}{k}=c(c-1)\cdots(c-k+1)/k!$ involve division by $p$ when $k\ge p$.
One way to find a useable substitute is by truncating the series before its term of degree $p$.
Limiting ourselves, for simplicity, to coefficients in the field of rational functions $\F_p(\alpha)$,
for each $f(\alpha)\in\F_p(\alpha)$ we define a polynomial
\[
(1+X)_\ast^{f(\alpha)}
=\sum_{k=0}^{p-1}\binom{f(\alpha)}{k}X^k
\in\F_p(\alpha)[X].
\]
Then for $f(\alpha),g(\alpha)\in\F_p(\alpha)$ we have
\[
(1+X)_\ast^{f(\alpha)+g(\alpha)}\equiv(1+X)_\ast^{f(\alpha)}(1+X)_\ast^{g(\alpha)}\pmod{X^p}
\]
in $\F_p(\alpha)[X]$.
Also, taking the derivative with respect to $X$ of such a product of truncated binomials,
or of variations such as $(1+bX)^{f(\alpha)}$, satisfies Leibniz rule modulo $X^p$.
(In formal algebraic terms this means that the standard derivation $d/dX$ of the polynomial ring $\F_p(\alpha)[X]/(X^p)$
induces a derivation of its quotient ring $\F_p(\alpha)[X]/(X^p)$,
because it maps the factored ideal $(X^p)$ into itself.)
However, we have
\[
\frac{d}{dX}
(1+X)_\ast^{f(\alpha)}
=f(\alpha)(1+X)_\ast^{f(\alpha)-1}+\bigl(f(\alpha)^p-f(\alpha)\bigr)X^{p-1}.
\]

With this definition, the expression
$b_{r,s}(\alpha)$ equals the coefficient of $X^{p-1}$ in the polynomial
$(1+X/r)_\ast^{r\alpha-1}(1-X/s)_\ast^{s\alpha-1}$.
As a simple application, we have
\[
b_{r,r}(\alpha)
=
(-1)^{\frac{p-1}{2}}\binom{r\alpha-1}{\frac{p-1}{2}},
\]
because this is the coefficient of $X^{p-1}$ in $(1-X^2/r^2)^{r\alpha-1}$.
A less trivial application gives the following alternate expression for the polynomials $b_{r,s}(\alpha)$.

\begin{lemma}\label{lemma:b(ra,sa)_alt}
If $r+s\neq p$ then
\begin{equation*}
b_{r,s}(\alpha)=\sum_{k=0}^{p-1}(-r/s)^k\binom{r\alpha -1}{p-1-k}\binom{s \alpha}{k}.
\end{equation*}
\end{lemma}

This expression for $b_{r,s}(\alpha)$ would have naturally come up in the proof
of Lemma~\ref{lemma:product}, had we compared the coefficients of $X^{0}$ rather than $X^{p-1}$ in the general case $r+s\neq p$ as well.

\begin{proof}
The lemma claims that $b_{r,s}(\alpha)$ equals the coefficient of $X^{p-1}$ in the polynomial
$(1+X/r)_\ast^{r\alpha-1}(1-X/s)_\ast^{s\alpha}$,
rather than in the polynomial
$(1+X/r)_\ast^{r\alpha-1}(1-X/s)_\ast^{s\alpha-1}$
as in the defining Equation~\eqref{eq:b(ra,sa)}.
Hence it suffices to show that the coefficient of $X^{p-1}$ is the same in those two polynomials.

We start with the polynomial
$(1+X/r)_\ast^{r\alpha}(1-X/s)_\ast^{s\alpha}$
and note that its derivative (with respect to $X$) has no term of degree $p-1$
(because the derivative of $X^p$ vanishes).
The derivative of this product equals
\begin{align*}
&
\bigl(\alpha(1+X/r)_\ast^{r\alpha-1}+(\alpha^p-\alpha)X^{p-1}\bigr)
\cdot(1-X/s)_\ast^{s\alpha}
\\&\qquad
-(1+X/r)_\ast^{r\alpha}
\cdot\bigl(\alpha(1-X/s)_\ast^{s\alpha-1}+(\alpha^p-\alpha)X^{p-1}\bigr).
\end{align*}
Consequently, there is no term of degree $p-1$ in
\[
\alpha
(1+X/r)_\ast^{r\alpha-1}(1-X/s)_\ast^{s\alpha}
-\alpha
(1+X/r)_\ast^{r\alpha}(1-X/s)_\ast^{s\alpha-1},
\]
which in turn equals
\[
\frac{(r+s)\alpha}{r}\cdot
(1+X/r)_\ast^{r\alpha-1}(1-X/s)_\ast^{s\alpha}
-\frac{(r+s)\alpha}{r}\cdot
(1+X/r)_\ast^{r\alpha-1}(1-X/s)_\ast^{s\alpha-1}.
\]
Because $r+s\neq p$ the desired conclusion follows.
\end{proof}

\section{Proofs of Theorem~\ref{thm:reciprocal} and Theorem~\ref{thm:powers}}\label{sec:proofs_B}

Exploiting several properties of the polynomials $b_{1,s}(\alpha)$ which we have proved in Section~\ref{sec:proofs}
we can finally prove the functional equation for the polynomials $G^{(\alpha)}(X)$ announced in Theorem~\ref{thm:reciprocal}.

\begin{proof}[Proof of Theorem~\ref{thm:reciprocal}]
We have
\begin{align*}
\frac{X^{p}}{\alpha^{p-1}-1}G^{(-\alpha)}\left(\frac{1-\alpha^{p-1}}{X}\right)
&=
\sum_{k=1}^{p-1}\frac{X^{p-k}}{k}
\frac{(1-\alpha^{p-1})^{k-1}}{\prod_{s=1}^{k-1}b_{1,s}(-\alpha)}
\\&=
\sum_{k=1}^{p-1}\left(\prod_{s=1}^{k-1}b_{1,s}(\alpha)\right)\frac{X^{p-k}}{k}
&&\text{by Lemma~\ref{lemma:b(a)b(-a)}}
\\&=
-\sum_{k=1}^{p-1}\left(\prod_{s=1}^{p-k-1}b_{1,s}(\alpha)\right)\frac{X^{k}}{k}
\\&=
-\sum_{k=1}^{p-1}\left(\prod_{s=k}^{p-2}b_{1,s}(\alpha)\right)\frac{X^{k}}{k}
&&\text{by Corollary~\ref{cor:symmetry}}
\\&=
-\prod_{r=1}^{p-2}b_{1,r}(\alpha)\cdot\sum_{k=1}^{p-1}\frac{1}{k}
\frac{X^{k}}{\prod_{s=1}^{k-1}b_{1,s}(\alpha)}
\\&=
\prod_{k=2}^{p-1}(1+\alpha/k)^{k-1}\cdot
G^{(\alpha)}(X)
&&\text{by Corollary~\ref{cor:product}.}
\end{align*}
The desired equation follows upon multiplication by
$1-\alpha^{p-1}=\prod_{k=1}^{p-1}(1+\alpha/k)$.
\end{proof}

We conclude this section with proving Theorem~\ref{thm:powers}.

\begin{proof}[Proof of Theorem~\ref{thm:powers}]
Quoting Equation~\eqref{eq:power}, we have
\[
\bigl(L_{p-1}^{(\alpha)}(X)\bigr)^h
\equiv \biggl(\prod_{s=1}^{h-1}b_{1,s}(\alpha)\biggr)\cdot L_{p-1}^{(h\alpha)}\bigl(hX\bigr)
\pmod{X^p-(\alpha^p-\alpha)},
\]
for $0<h<p$.
In essence, we now would like to apply $G^{(h\alpha)}$ to both sides after dividing them by $\prod_{s=1}^{h-1}b_{1,s}(\alpha)$,
and then regard $L_{p-1}^{(\alpha)}(X)$ as the new variable.
However, some care is required to put the change of variable on a rigorous ground,
as we are dealing with congruences.

As in the proof of Corollary~\ref{cor:inverse}, we extend the field of coefficients of the polynomial ring to $\F_p(\alpha^{1/p})$,
so that the modulus $X^p-(\alpha^p-\alpha)$ becomes a $p$th power.
Rename the indeterminate $X$ to $x$, view $L_{p-1}^{(\alpha)}(x)$
as a polynomial in $\F_p(\alpha^{1/p})[x]$,
replace $x$ with $x+\alpha-\alpha^{1/p}$,
and divide both sides by $\prod_{s=1}^{h-1}b_{1,s}(\alpha)$.
The above congruence then reads
\[
\frac{\bigl(L_{p-1}^{(\alpha)}(x+\alpha-\alpha^{1/p})\bigr)^h}{\prod_{s=1}^{h-1}b_{1,s}(\alpha)}
\equiv L_{p-1}^{(h\alpha)}\bigl(h(x+\alpha-\alpha^{1/p})\bigr)
\pmod{x^p}
\]
in the polynomial ring $\F_p(\alpha^{1/p})[x]$, but we rather view it in the power series ring $\F_p(\alpha^{1/p})[[x]]$.
In this setting we can apply the change of indeterminate (or uniformizing parameter)
\[
\delta-X=L_{p-1}^{(\alpha)}(x+\alpha-\alpha^{1/p}),
\]
where $\delta=L_{p-1}^{(\alpha)}(\alpha-\alpha^{1/p})\in\F_p(\alpha^{1/p})$.
According to Theorem~\ref{thm:inverse} we then have
\[
G^{(\alpha)}(\delta-X)\equiv x+\alpha-\alpha^{1/p}\pmod{X^p}
\]
in the power series ring $\F_p(\alpha^{1/p})[[X]]=\F_p(\alpha^{1/p})[[x]]$.

Now we are in a position to apply $G^{(h\alpha)}$ to both sides of our congruence,
and taking Theorem~\ref{thm:inverse} into account (with $h\alpha$ in place of $\alpha$
and $h(x+\alpha-\alpha^{1/p})$ in place of $X$) we get
\[
G^{(h\alpha)}\left(\frac{(\delta-X)^h}{\prod_{s=1}^{h-1}b_{1,s}(\alpha)}\right)
\equiv hG^{(\alpha)}(\delta-X)
\pmod{X^p}
\]
in $\F_p(\alpha^{1/p})[[X]]$.
However, this congruence actually takes place in the polynomial ring $\F_p(\alpha^{1/p})[X]$,
and then replacing $X$ with $\delta-X$ we find
\[
G^{(h\alpha)}\left(\frac{X^h}{\prod_{s=1}^{h-1}b_{1,s}(\alpha)}\right)
\equiv hG^{(\alpha)}(X)
\pmod{X^p-\delta^p},
\]
which takes place in $\F_p(\alpha)[X]$ because $\delta^p=L_{p-1}^{(\alpha^p)}(\alpha^p-\alpha)$, and is the desired conclusion.
\end{proof}

\section{A connection with Jacobi polynomials}\label{sec:Jacobi}

In this section we discuss how our polynomials $b_{r,s}(\alpha)$ of Equation~\eqref{eq:b(ra,sa)}
relate to certain Jacobi polynomials viewed modulo $p$.
Standard definitions and properties of Jacobi polynomials can be found in the the classical book by Szeg\"o on orthogonal polynomials~\cite{Szego:orthogonal}, but
readers should note that purely combinatorial proofs of those properties were presented in~\cite{LerouxStrehl}.

The classical $n$-th Jacobi polynomial $P_{n}^{(\alpha, \beta)}(x)$, for $n$ a non-negative integer, are given by
\[
P_{n}^{(\alpha, \beta)}(x)=\frac{1}{2^n}\sum_{k=0}^{n}\binom{\alpha+n}{n-k}\binom{\beta+n}{k}(x+1)^{n-k}(x-1)^{k}.
\]
Here $\alpha$ and $\beta$ are parameters, but because
$P_{n}^{(\alpha,\beta)}(x)$ depends polynomially on them we may also view it as a polynomial in three indeterminates,
$P_{n}^{(\alpha,\beta)}(x)\in\Q[\alpha,\beta,x]$.

For $p$ an odd prime, the coefficients of the polynomials $P_{p-1}^{(\alpha, \beta)}(x)$ are $p$-integral rational numbers,
and hence they may be viewed modulo $p$.
Because $2^{p-1}\equiv 1\pmod{p}$ we have
\[
P_{p-1}^{(\alpha, \beta)}(x)\equiv \sum_{k=0}^{p-1}\binom{\alpha-1}{p-1-k}\binom{\beta-1}{k}(x+1)^{p-1-k}(x-1)^{k} \pmod{p}.
\]
Setting $x=(s-r)/(s+r)$ for $0<r,s<p$ with $r+s\neq p$ we find
\[
b_{r,s}(\alpha)
\equiv
P_{p-1}^{(r\alpha, s\alpha)}\bigl((s-r)/(s+r)\bigr)
\pmod{p},
\]
where we are slightly abusing notation as the right-hand side is a polynomial with $p$-integral coefficients
while $b_{r,s}(\alpha)$ has coefficients in $\F_p$ by our definition.
Keeping the same abuse of notation we will now exploit this connection to recover some of the properties of the polynomials $b_{r,s}(\alpha)$.

Beyond the obvious identity
$
P_{n}^{(\alpha, \beta)}(x)=(-1)^{n} P_{n}^{(\beta, \alpha)}(-x),
$
which yields the equally obvious symmetry
$
b_{r,s}(\alpha)=b_{s,r}(\alpha),
$
Jacobi polynomials satisfy
\[
(2p+\alpha +\beta)\cdot\frac{x+1}{2}\cdot P_{p-1}^{(\alpha, \beta +1)}(x)=(p+\beta)P_{p-1}^{(\alpha, \beta)}(x)+pP_{p}^{(\alpha, \beta)}(x)
\]
for all positive integers $p$.
This identity is~\cite[Equation~(4.5.4)]{Szego:orthogonal}, where we have written $p-1$ in place of the customary $n$.
With $p$ an odd prime as in the present context, note that the coefficients of $P_{p}^{(\alpha, \beta)}(x)$ are not all $p$-integral,
but those of $pP_{p}^{(\alpha, \beta)}(x)$ are, and in fact
\begin{align*}
pP_{p}^{(\alpha, \beta)}(x)
&=
\frac{p}{2^p} \sum_{k=0}^{p}\binom{\alpha+p}{p-k}\binom{\beta+p}{k}(x+1)^{p-k}(x-1)^{k}
\\&\equiv
\frac{p}{2}\binom{\alpha+p}{p}(x+1)^{p} +\frac{p}{2}\binom{\beta+p}{p}(x-1)^{p}
\pmod{p}
\\
&\equiv \tfrac{1}{2}(\alpha -\alpha^{p})(x+1)^{p} +\tfrac{1}{2}(\beta -\beta^{p})(x-1)^{p}
\pmod{p}.
\end{align*}
Setting $x=(s-r)/(s+r)$, and replacing $\alpha$ and $\beta$ with $r\alpha$ and $s\alpha$, respectively, we find
\[
P_{p-1}^{(r \alpha, s \alpha+1)}\left(\frac{s-r}{s+r}\right) \equiv
P_{p-1}^{(r \alpha, s \alpha)}\left(\frac{s-r}{s+r}\right) \pmod{p},
\]
whose left-hand side corresponds to the alternate expression for $b_{r,s}(\alpha)$ which we found in Lemma~\ref{lemma:b(ra,sa)_alt}.

Finally, Jacobi polynomials satisfy
\[
P_{n}^{(\alpha, \beta)}(x)= \left(\frac{x+1}{2}\right)^{n}  P_{n}^{(\alpha, -2n -\alpha -\beta -1)}\left( \frac{3-x}{x+1}\right),
\]
which is~\cite[Equation~(4.22.1)]{Szego:orthogonal}.
With the appropriate specializations as above this yields the first of the following congruences
\[
P_{p-1}^{( \alpha, s \alpha)}\left(\frac{s-1}{s+1}\right) \equiv
P_{p-1}^{( \alpha, (-s-1) \alpha+1)}\left(\frac{s+2}{s}\right) \equiv
P_{p-1}^{( \alpha, (-s-1) \alpha)}\left(\frac{s+2}{s}\right)   \pmod{p},
\]
and the second one is the one found above corresponding to Lemma~\ref{lemma:b(ra,sa)_alt}.
Altogether, we have recovered that
$
b_{1,s}(\alpha)=b_{1,-1-s}(\alpha)
$
for $0<s<p-1$, which is our Corollary~\ref{cor:symmetry}.

\bibliography{References}

\end{document}